\documentclass[a4paper,11pt]{article}

\usepackage{amsmath}
\usepackage{booktabs}
\usepackage{amssymb}
\usepackage{enumerate}
\usepackage{amsopn}
\usepackage{amsthm}
\usepackage{booktabs}
\usepackage{pdfpages}
\usepackage{float}
\usepackage{mathrsfs}
\usepackage[dvipsnames]{xcolor}
\usepackage[margin=30mm]{geometry}
\usepackage{graphicx}
	\graphicspath{{./}}
\usepackage{hyperref}
	\hypersetup{
		colorlinks=true,
		linkcolor=blue,
		filecolor=blue,
		urlcolor=blue,
		citecolor=blue
	}

\usepackage{soul}
\usepackage{url}

\usepackage{setspace}
\newtheorem{thm}{Theorem}[section]
\newtheorem{cor}[thm]{Corollary}
\newtheorem{lem}[thm]{Lemma}
\newtheorem{prop}[thm]{Proposition}

\theoremstyle{definition}
\newtheorem{defn}[thm]{Definition}
\newtheorem{example}[thm]{Example}

\theoremstyle{remark}
\newtheorem{remark}[thm]{Remark}
\newtheorem{claim}[thm]{Claim}

\DeclareMathOperator{\Type}{type}

\DeclareMathOperator{\Col}{Col} 
\DeclareMathOperator{\Hom}{Hom} 
\DeclareMathOperator{\Op}{Op}

\newcommand{\Z}{\mathbb{Z}}

\newcommand{\N}{\mathbb{N}}
\newcommand{\CP}{\mathbb{C}P}
\newcommand{\K}{\mathcal{K}}
\newcommand{\D}{\mathcal{D}}

\numberwithin{equation}{section}

\usepackage{color}
\usepackage{url}
\definecolor{darkgreen}{cmyk}{1,0,1,.2}
\definecolor{darkorchid}{rgb}{1.0, 0.5, 0}
\definecolor{persimmon}{rgb}{0.93, 0.35, 0.0}

\newdimen\theight
\def\TeXref#1{%
             \leavevmode\vadjust{\setbox0=\hbox{{\tt
                     \quad\quad  {\small \textrm #1}}}%
             \theight=\ht0
             \advance\theight by \lineskip
             \kern -\theight \vbox to
             \theight{\rightline{\rlap{\box0}}%
             \vss}%
             }}%

\title{Quandle presentations of surface knots in $4$-manifolds and bridge numbers}

\author{Zhou Xiaozhou\thanks{Department of Mathematical Sciences, Ritsumeikan University, 1-1-1 Nojihigashi, Kusatsu, Shiga 525-8577, Japan. 
E-mail: \texttt{ra0108ik@ed.ritsumei.ac.jp}}}

\date{}

\begin{document}

\maketitle


\begin{abstract}
The fundamental quandle is an invariant for distinguishing surface knots, yet 
computable presentations have traditionally been limited to surfaces embedded in 
the $4$-sphere. Building on the framework of banded unlink diagrams introduced 
by Hughes, Kim, and Miller, we give a Wirtinger type presentation of the fundamental quandle of surface links in arbitrary $4$-manifolds. As applications, we extend the work of Sato and Tanaka to show that for any $b \geq 4$ and $m \geq 0$, there exist infinitely many pairwise non-local surface knots with bridge number $b$ in $\mathbb{C}P^2 \#m\overline{\mathbb{C}P^2}$, and we distinguish infinite families of surface knots with isomorphic knot groups, extending results of Tanaka and Taniguchi.
\end{abstract}

\section{Introduction}

The knot group, defined as the fundamental group of the complement, has traditionally served as a fundamental invariant in the study of surface knots. However, the knot group is not a complete invariant, as there exist non-isotopic surface knots with isomorphic knot groups. To capture finer topological information, Joyce and Matveev independently introduced the algebraic structure of a quandle in the 1980s \cite{joyce1982knotquandle, matveev1984distributive}. This theory was further developed by Fenn and Rourke, who defined racks and augmented quandles, establishing them as essential invariants for codimension-two knots \cite{fenn1992racks}.

In the case of surface knots embedded in the $4$-sphere $S^4$, the theory is well-established. These surfaces can be effectively described by various diagrammatic methods, such as broken surface diagrams, motion pictures, and marked graph diagrams (ch-diagrams)~\cite{carter2013surfaces}. Based on these descriptions, combinatorial algorithms for deriving Wirtinger-type presentations of the fundamental quandle have been extensively developed. For instance, Kamada, Kim, and Lee provided a method for computing quandle invariants from marked graph diagrams~\cite{kamada2015computations}, and Ashihara developed algorithms to calculate fundamental biquandles from ch-diagrams~\cite{ashihara2012calculating}. These computational tools enable the extraction of algebraic data directly from the geometric descriptions of surfaces in $S^4$.

However, for surface knots embedded in general $4$-manifolds, such as the complex projective plane $\mathbb{C}P^2$, classical diagrammatic methods are often insufficient or difficult to apply. To investigate surface knots in arbitrary $4$-manifolds, two powerful and complementary frameworks have recently emerged: banded unlink diagrams and bridge trisections.

Hughes, Kim, and Miller introduced banded unlink diagrams, which represent surfaces by using Kirby diagram of $4$-manifold~\cite{hughes2020isotopies}, providing a framework to extract algebraic invariants directly from diagrams. Complementing this, Meier and Zupan developed the theory of bridge trisections~\cite{meier2018bridge}, extending the work of Gay and Kirby~\cite{gay2016trisecting}.

In this paper, we establish a diagrammatic method to compute the fundamental quandle of a surface link $S$ embedded in an arbitrary $4$-manifold. For a banded unlink diagram $\mathcal{D}$ of $S$, we define a diagrammatic quandle $Q(\mathcal{D})$ by encoding the crossing information and the band data of the diagram. We prove that this algebraically defined quandle is isomorphic to the topological fundamental quandle $Q(S)$, thereby providing a computable Wirtinger type presentation for surface knots in general $4$-manifolds. Our main result is stated as follows:

\begingroup
    \newtheoremstyle{boldheader}
    {}
    {}
    {\itshape}
    {}
    {\bfseries}
    {.}
    { }
    {\thmname{#1}\thmnote{ #3}}

    \theoremstyle{boldheader}
    \newtheorem*{thmRepeat}{Theorem \ref{thm:qdle_presentation}} 
    
    \begin{thmRepeat}
        Let $S$ be a surface link and $\mathcal{D}$ be a banded unlink diagram of $S$. Then $Q(S)$ and $Q(\mathcal{D})$ are isomorphic. 
    \end{thmRepeat}
\endgroup

As a consequence of Theorem~\ref{thm:qdle_presentation}, the fundamental kei of $S\# C_d$ can be described explicitly in terms of the parity of $d$ (see Corollary~\ref{cor:K(S+C_d)}). 

As the first application, we use kei colorings, namely quandle homomorphisms 
into finite kei, to study the bridge number of surface knots. This extends the 
work of Sato and Tanaka on surface knots in $S^4$ \cite{sato2022bridge} to 
arbitrary $4$-manifolds without $1$-handles for local surface knots.
\begingroup
    \newtheoremstyle{boldheader}{}{}{\itshape}{}{\bfseries}{.}{ }
    {\thmname{#1}\thmnote{ #3}}
    \theoremstyle{boldheader}
    \newtheorem*{corRepeat1}{Corollary \ref{cor:local_infinite_many}} 
    \begin{corRepeat1}
        Let $X$ be a closed $4$-manifold which admits a handle decomposition 
        without $1$-handle. For any integer $b\geq4$, there exist infinitely 
        many pairwise non-isotopic local embedded surfaces with bridge number 
        $b$ in $X$.
    \end{corRepeat1}
\endgroup
Furthermore, by constructing examples of surfaces which are not null-homologous, 
we obtain the following result for the non-local case:
\begingroup
    \newtheoremstyle{boldheader}{}{}{\itshape}{}{\bfseries}{.}{ }
    {\thmname{#1}\thmnote{ #3}}
    \theoremstyle{boldheader}
    \newtheorem*{thmRepeat1}{Theorem \ref{thm:infinitely_many_bridge}} 
    \begin{thmRepeat1}
        In $\CP^2\#m\overline{\CP}^2$, for any integer $b\geq4$ and $m\geq0$, 
        there exist infinitely many pairwise homologous non-isotopic embedded 
        surfaces with bridge number $b$.
    \end{thmRepeat1}
\endgroup

As the second application, we generalize known phenomena in $S^4$ concerning 
surface knots with identical knot groups but distinct fundamental quandles 
\cite{tanaka20252} to arbitrary $4$-manifolds without $1$-handles for local 
surface knots.
\begingroup
    \newtheoremstyle{boldheader}{}{}{\itshape}{}{\bfseries}{.}{ }
    {\thmname{#1}\thmnote{ #3}}
    \theoremstyle{boldheader}
    \newtheorem*{corRepeat2}{Corollary \ref{cor:quandle_fund_group}} 
    \begin{corRepeat2}
        Let $X$ be a $4$-manifold without $1$-handle. There exist infinitely 
        many triples of pairwise non-isotopic local surfaces $\{S_1,S_2,S_3\}$ 
        embedded in $X$ such that 
        \begin{enumerate}[(1)]
            \item the knot groups $G(S_1), G(S_2), G(S_3)$ are isomorphic 
            to each other, 
            \item the knot quandles $Q(S_1),Q(S_2),Q(S_3)$ are pairwise 
            non-isomorphic. 
        \end{enumerate}
    \end{corRepeat2}
\endgroup
Furthermore, by constructing examples of surfaces which are not null-homologous, 
we obtain the following result for the non-local case:
\begingroup
    \newtheoremstyle{boldheader}{}{}{\itshape}{}{\bfseries}{.}{ }
    {\thmname{#1}\thmnote{ #3}}
    \theoremstyle{boldheader}
    \newtheorem*{thmRepeat2}{Theorem \ref{thm:infinite_qdl_CP2}} 
    \begin{thmRepeat2}
        For every integer $m\geq0$, there exist infinitely many triples of 
        pairwise homologous surfaces $\{S_1,S_2,S_3\}$ embedded in 
        $\CP^2\#m\overline{\CP}^2$ which are not local, such that 
        \begin{enumerate}[(1)]
            \item the knot groups $G(S_1), G(S_2), G(S_3)$ are isomorphic;
            \item the knot quandles $Q(S_1),Q(S_2),Q(S_3)$ are pairwise 
            non-isomorphic. 
        \end{enumerate}
    \end{thmRepeat2}
\endgroup

\subsection*{Organization of the paper}

In section \ref{sec:preliminaries}, we review the basic definitions of quandle, augmented quandle, and the fundamental quandle of a surface knot. Section \ref{sec:diagrammatic} establishes a diagrammatic method for computing the fundamental quandle. We recall the notion of banded unlink diagram due to \cite{hughes2020isotopies} and define the diagrammatic quandle $Q(\D)$ associated with such a diagram. The main result of this section, Theorem~\ref{thm:qdle_presentation}, assert that $Q(\D)$ is isomorphic to the topological fundamental quandle $Q(S)$. In Section \ref{sec:applications}, we apply these invariants to study the bridge numbers of surface knots and to distinguish surface knots with isomorphic knot groups.


\section{Preliminary}\label{sec:preliminaries}

\subsection{Quandles}

Let us begin by recalling the definitions of quandles and augmented quandles, following \cite{fenn1992racks}.

\begin{defn}
    A \textit{quandle} is a set $Q$ equipped with a binary operation $(x,y) \mapsto x^y$ which satisfies the following three axioms:
	\begin{enumerate}[(1)]
		\item For any $x$ in $Q$, we have $x^x=x$.
		\item For any $x,y$ in $Q$, there exists a unique element $z$ in $Q$ such that $x=z^y$.
		\item For any $x,y,z$ in $Q$, the identity $(x^y)^{z}=(x^{z}){}^{y^z}$ holds.
	\end{enumerate}
    For simplicity, we denote $(x^y)^z$ by $x^{yz}$, and denote $x^{(y^z)}$ by $x^{y^z}$. 
    A quandle $Q$ is called a \textit{kei} if $x^{yy} = x$ for all $x,y$ in $Q$.
\end{defn}

\begin{example}
	A \textit{dihedral quandle} $R_n$ is the set $R_n=\Z/ n\Z$ equipped with the operation $x^y:=2y-x$. A direct computation shows that $x^{yy}=x$ for all $x,y$ in $R_n$. Hence, $R_n$ is a kei. 
\end{example}

By the second axiom, for each $y\in Q$, we can define a bijection $f_y:Q\to Q$ by $f_y(t)=t^y$. We denote its inverse by $f_y^{-1}$. We shall write $x^{\bar y}=f_y^{-1}(x)$. Note that $\bar y$ is not an element of $Q$. We just consider the mapping $x \mapsto x^{\bar{y}}$, which is called an operator. With this notation, the third axiom can be rewritten in the following equivalent form:
\begin{align*}
    x^{y^z} = x^{\bar{z} y z}, \quad \forall x, y, z \in Q.
\end{align*}
This leads us to regard each element of a quandle as an operator. 

In expressions such as $x^{y\bar z}$, we refer to $x$ as being at the \textit{primary level}, and to $y,\bar z$ as being at the \textit{operator level}.
By the third axiom, any expression obtained by iterated quandle operations can be written in the form $x^w$, where $x$ in $Q$ is at the primary level and $w$ is a word in the free group $F(Q)$ on $Q$, at the operator level. From this viewpoint, the operation of quandle could be interpreted into an action of $F(Q)$ on $Q$, written as $(x,g)\mapsto x\cdot g$. 

Two elements $x,y\in Q$ are said to be \textit{operator equivalent} if $f_x=f_y$, and we write  $x\equiv y$. 
The corresponding equivalence classes define the \textit{operator group} $\Op(Q)$ of the quandle.
 More precisely, $\Op(Q)$ is defined as the quotient $F(Q)/N$, where $F(Q)$ is the free group generated by the element of $Q$, and $N$ is the normal subgroup
$$N=\{w\in F(Q)\mid w\equiv1\}.$$

This operator interpretation motivates a natural variant of quandles,
in which the operator group is made explicit.

\begin{defn}
    An \emph{augmented quandle} is a pair $(X,G)$ together with the following data:
    \begin{itemize}
      \item a right action of the group $G$ on the set $X$, written as $(x,g)\mapsto x^g$;
      \item the conjugation action of $G$ on itself, given by $g^h = h^{-1}gh$;
      \item a map $\partial : X \to G$, called the \emph{augmentation map},
    \end{itemize}
    such that the following \emph{augmentation identity} holds:
    \[
    \partial(x^g) = g^{-1}\partial(x)g
    \quad \text{for all } x \in X \text{ and } g \in G.
    \]
\end{defn}

Given an augmented quandle $(X,G)$, the induced quandle operation on $X$ is defined by letting $x^y$ be the result of acting on $x$ by $\partial(y) \in G$. The augmentation identity then implies that
\[
\partial(x^y)=\partial(y)^{-1}\partial(x)\partial(y),
\]
which shows that the quandle operation is compatible with the conjugation action in $G$.

From now on, we shall use the notation $x^g$ to denote the action of $g\in G$ on $x\in X$,
and, whenever an element $y\in X$ appears at the operator level,
identify it with its image $\partial(y)\in G$.

We introduce the general quandle presentations, following the definition in \cite{fenn1992racks}. 

\begin{defn}
	Let $Q$ be a quandle. An equivalence relation $\sim$ on $Q$ is called a \textit{congruence} if, for all $x,y,z,w\in Q$, the condition $x\sim y, z\sim w$ implies $x^z\sim y^w$. 
\end{defn}

\begin{defn}
    Let $S$ and $T$ be two sets, and denote by $F(S \cup T)$ the free group on $S \cup T$. 
    The \textit{extended free quandle}, denoted by $FQ(S, T)$, is the augmented quandle defined on the set $S \times F(S \cup T)$ with the acting group $F(S \cup T)$.
    
    Following standard notation, we denote the element $(x, w)$ in $S \times F(S \cup T)$ by $x^w$. The augmentation map $\partial: FQ(S, T) \to F(S \cup T)$ is given by $\partial(x^w):=w^{-1}xw$. The associated quandle operation is given by
    \[
    (x^w)^{(y^z)} = (x^w)^{\partial(y^z)} = x^{wz^{-1}yz}
    \]
    for any $x,y\in S$ and $w,z\in F(S \cup T)$.
\end{defn}

\begin{defn}
    A \textit{general augmented quandle presentation} is a quadruple $ [S_P, S_O \mid R_P, R_O] $, consisting of the following four sets:
    \begin{itemize}
        \item $S_P$: a set of \textit{primary generators};
        \item $S_O$: a set of \textit{operator generators};
        \item $R_P$: a set of \textit{primary relations}, which are equations of the form $x=y$ with $x,y$ are elements of $FQ(S_P, S_O)$;
        \item $R_O$: a set of \textit{operator relations}, which are relations in the group $F(S_P \cup S_O)$.
    \end{itemize}
    The augmented quandle defined by such a presentation is the quotient 
    \[
    Q := FQ(S_P, S_O) / \sim,
    \]
    where $\sim$ is the smallest congruence on $FQ(S_P, S_O)$ satisfying:
    \begin{enumerate}
        \item $x \sim y$ whenever $(x=y)$ is in $R_P$;
        \item $z^{\partial(x)} \sim z^{\partial(y)}$ (or simply $z^x \sim z^y$) for any $z$ in $FQ(S_P, S_O)$, whenever $(x=y)$ is in $R_P$;
        \item $z^u \sim z^v$ for any $z$ in $FQ(S_P, S_O)$, whenever $(u=v)$ is in $R_O$.
    \end{enumerate}
    If $S_O = \emptyset$, the presentation is called a \emph{primary quandle presentation}.
\end{defn}

\begin{remark}\label{rmk}
    When the operator group is finitely presented, the operator relations $R_O$ often arise from a presentation of the operator group. More precisely, suppose that the operator group has a presentation $\langle S_O \mid R_G\rangle$. Via the augmentation map $\partial:Q\to G$, each relation in $R_G$ induces an operator relation $u\equiv v$ in $R_O$. 
\end{remark}


\subsection{Fundamental augmented quandles of surface knots in $4$-manifolds}

The definition of fundamental quandle of a link is introduced in \cite{fenn1992racks}. Let $S$ be a surface link embedded in the $4$-manifold $X$. 
Let $\nu(S)$ be the tubular neighborhood of $S$. Define $E(S)=\overline{X\setminus \nu(S)}$ which is called the exterior of $S$. 
Fix a base point $*\in E(S)$. We denote the fundamental group $\pi_1(E(S))$ by $G(S)$ and call it the \textit{link group} or \textit{knot group}.

Let $N=(D,\xi)$ be a pair consisting of a meridian disk $D$ of $S$ and a path $\xi$ in $E(S)$ starting at a point on the boundary $\partial D$ and terminating at the basepoint $x_0$.
Such a pair is called a \textit{noose}. We denote the homotopy class of $N$ by $[N]=[(D,\xi)]$, and let $Q(S)$ denote the set of all such homotopy classes. 

We define an action of $G(S)$ on $Q(S)$. Let $x\in Q(S)$ be represented by the path $\xi$, and let $g\in G(S)$ be represented by a loop $\gamma$ based at $*$. We define
$$x\cdot g := [\xi \cdot \gamma]$$
where the right-hand $\cdot$ denotes the concatenation of paths. 

The set $Q(S)$ may equipped with a structure of an augmented quandle by using the action of $G(S)$.

\begin{defn}
	The \textit{fundamental augmented quandle} of a surface link $S$ is an augmented quandle $(Q(S),G(S))$. Let $x,y$ in $Q(S)$ be represented by the $N_x=(D_x,\xi)$ and $N_y=
    (D_y,\eta)$, respectively. The augmentation map $\partial:Q(S)\to G(S)$ is defined by the product of paths
	$$\partial(x) := \bar{\xi} \, m_x \, \xi$$
    where $m_x = \partial D_x$ is the boundary loop of the meridian disk $D_x$. This augmentation induces a quandle operation on $Q(S)$: the operation $x^y$ is defined as the action of $\partial(y)$ on $x$, given explicitly by
    $$x^y = [(D_x, \xi \ \bar{\eta}\  m_y \ \eta)].$$
\end{defn}

\textbf{Notation Convention:} Since the operator group is canonically the knot group $G(S)=\pi_1(E(S))$, we will henceforth denote the fundamental augmented quandle simply by $Q(S)$, with the understanding that it carries this specific augmented structure.


\section{Wirtinger type presentation of surface knot quandles}\label{sec:diagrammatic}

\subsection{Banded unlink diagrams of surface links in $4$-manifolds}

We first recall the definition of the banded unlink diagram, following \cite{hughes2020isotopies}. 

Let $S$ be a surface knot in a $4$-manifold $X$. Let $h:X\to [0,4]$ be the self-indexing Morse function of $X$. Then we can construct a Kirby diagram $\mathcal{K}=L_1\cup L_2$ associated with the Morse function $h$, where $L_1$ is a dotted unlink associated with the $1$-handle attachment, and $L_2$ is an unlink with integer framing, associated with the $2$-handle attachment. Let $E(\mathcal{K})=S^3\setminus (\nu(L_1)\cup \nu(L_2))$ be the exterior of the Kirby diagram. 

Up to isotopy, the surface $S$ admits the following movie description. At each height $t\in[0,4]$, let $X_t=h^{-1}(t)$ denote the level set.

\begin{itemize}
    \item At $t=\frac{1}{2}$, $S\cap X_{1/2}$ consists of a collection of disjoint embedded disks, corresponding to the index-$0$ critical points of $S$. For every $t\in(\frac{1}{2},\frac{3}{2})$, the intersection $S\cap X_t$ is an unlink in $X_t$, which we denote by $L$.
    
    \item At $t=1$, we encounter the index-$1$ critical points of the ambient $4$-manifold $X$.
    
    \item At $t=\frac{3}{2}$, the $1$-handles of $S$ appear, and $S\cap X_{3/2}$ consists of the unlink $L$ together with a collection of disjoint bands attached to $L$. We denote this collection of bands by $v$. Resolving these bands produces a new unlink $L_v$, which describes $S\cap X_t$ for $t\in(\frac{3}{2},\frac{5}{2})$.
    
    \item At $t=2$, we encounter the index-$2$ critical points of $X$.
    
    \item At $t=\frac{5}{2}$, $S\cap X_{5/2}$ consists of a collection of disjoint embedded disks, corresponding to the index-$2$ critical points of $S$.
\end{itemize}

A surface link $S$ in $X$ admitting such a description is said to be in \textit{banded unlink position}. Let $S$ in $X$ be a surface link in banded unlink position. A \textit{banded unlink diagram} of $(X,S)$ is a triple $\mathcal{D}=(\mathcal{K},L,v)$, where 
\begin{itemize}
    \item $\mathcal{K}$ is a Kirby diagram of $X$,
    \item $L$ is an unlink in $E(\mathcal{K})$, 
    \item $v$ is a finite collection of bands attached to $L$ in $E(\mathcal{K})$,
\end{itemize}
 such that $L$ bounds a collection of disjoint embedded disks in $X_{1/2}$, and the resolved link $L_v$ bounds a collection of disjoint embedded disks in $X_{5/2}$. 

\begin{remark}
    We emphasize that the term ``unlink'' in this context refers to the topological property that the link bounds a collection of disjoint embedded disks in the 4-manifold (specifically in $X_{5/2}$), rather than its geometric appearance in the diagram.
\end{remark}

\begin{thm}[\cite{hughes2020isotopies}]
    Let $X$ be a $4$-manifold with Kirby diagram $\mathcal{K}$. Suppose that $S$ and $S'$ are embedded surfaces in $X$, with banded unlink diagrams $\mathcal{D}=(\mathcal{K}, L, v)$ and $\mathcal{D}'=(\mathcal{K}', L', v')$, respectively. Then $S$ and $S'$ are isotopic if and only if $\mathcal{D}$ and $\mathcal{D}'$ are related by a finite sequence of band moves, as described in Figure~\ref{fig:band_move}.
\end{thm}

\begin{figure}
    \centering
    \includegraphics[width=0.9\linewidth]{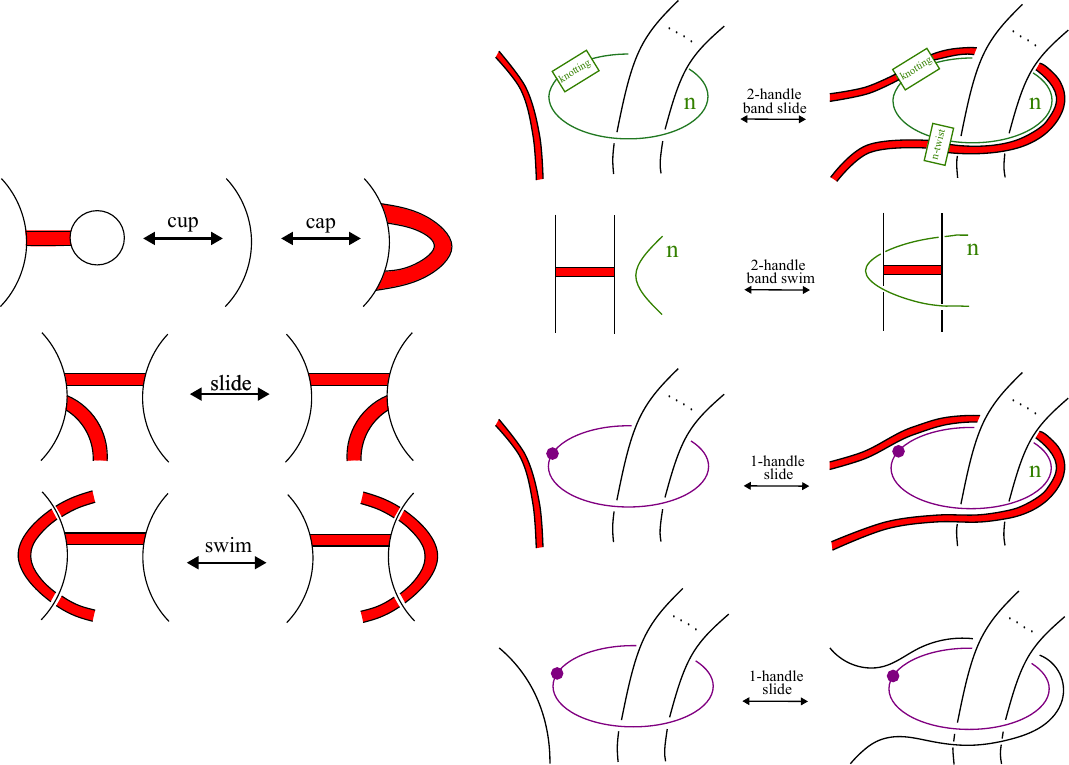}
    \caption{Band moves}
    \label{fig:band_move}
\end{figure}

We aim to compute the knot group $G(S)$ of a surface $S$ via a Kirby diagram of its exterior $E(S)$. Such a Kirby diagram $\mathcal{K}_S$ of the knot exterior $E(S)$ in $X$ can be constructed from a banded unlink diagram $\mathcal{D}$ using the method introduced in \cite{gompf19994}.

More generally, let $X^n$ be a smooth manifold equipped with a smooth function $t:X\to I=[0,1]$, and let $Y^m$ in $X$ be a properly embedded smooth submanifold such that $t|_Y:Y\to I$ is a Morse function. For $c\in I$, set $X_{\leq c}:=t^{-1}([0,c])\setminus \nu\bigl(Y\cap t^{-1}([0,c])\bigr)$, where $\nu(\cdot)$ denotes an open tubular neighborhood.

\begin{prop}[\cite{gompf19994}]
    Let $[a,d]$ with $a=0$ be an interval such that:
    \begin{enumerate}
     \item $t$ has no critical points in $t^{-1}((a,d))$,
     \item $t|_Y$ has a unique critical value $b\in(a,d)$, corresponding to a unique critical point $p\in Y$ of index $k$,
     \item $t^{-1}(a)$ and $t^{-1}(d)$ are regular level sets, and $Y$ meets these level sets transversely (equivalently, $t|_Y$ is regular at $a$ and $d$).
    \end{enumerate}
    Then $X_{\leq d}$ is obtained from $X_{\leq a}$ by attaching a single $(k+n-m-1)$-handle. In particular, $X_{\leq d} \cong X_{\leq a} \cup (k+n-m-1)\text{-handle}$. 
\end{prop}

Let $S$ be a surface embedded in a $4$-manifold $X$, and let $\mathcal{D}$ be a banded unlink diagram of $(X,S)$. Each $0$-handle of $S$ gives rise to a $1$-handle in the handle decomposition of the exterior $E(S)$. In the diagram, this $1$-handle is represented by a dotted circle, which is obtained by replacing the corresponding unlink component in $\mathcal{D}$.
Each $1$-handle of $S$ corresponds to a $2$-handle in the handle decomposition of the exterior $E(S)$. 
We assume that every $1$-handle (band) of $S$ has the blackboard framing. To visualize the construction of these $2$-handles, we consider the removal of a neighborhood of each band from the intermediate manifold. This operation effectively drills tunnels in the boundary. Consequently, each $2$-handle $h$ of $E(S)$ is attached along a curve that runs parallel to the core of the corresponding band (conceptually, consisting of two copies of the core, one running in front of the band and the other behind it, connected at the ends). Since we assumed the blackboard framing, the framing of this $2$-handle is $0$.

The framing of each such $2$-handle is $0$, which can be verified by drawing a parallel copy of the core of $h$. We add a $3$-handle for each $2$-handle $h$, and then add a $4$-handle to finish the picture. For an example, see Figure~\ref{fig:BUD->Kirby}. 

\begin{figure}[H]
    \centering
    \includegraphics[width=0.5\linewidth]{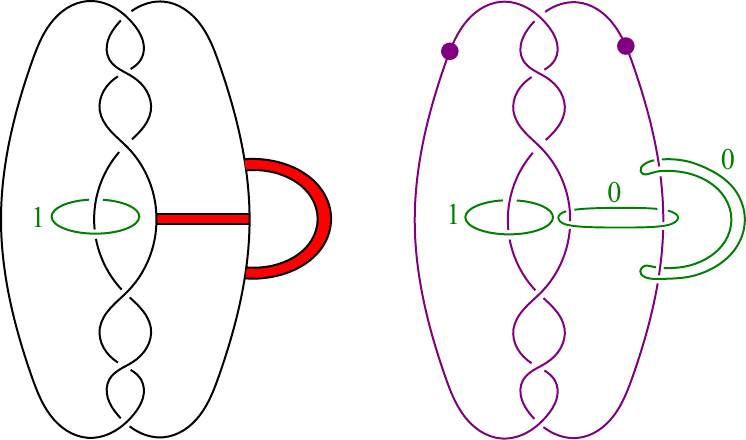}
    \caption{The left-hand figure depicts a banded unlink diagram of the connected sum of the spun trefoil and $\mathbb{C}P^1$ in $\mathbb{C}P^2$. By converting the unlink components into dotted circles and the bands into framed links, we obtain a Kirby diagram for the exterior of this surface in $\mathbb{C}P^2$, as illustrated in the right-hand figure.}
    \label{fig:BUD->Kirby}
\end{figure}

\subsection{A Wirtinger type presentation of the surface knot group}

For a surface link $S$ in a $4$-manifold $X$, the surface knot group $G(S) := \pi_1(E(S))$ is the fundamental group of the exterior $E(S) = X \setminus \text{int}(N(S))$. To compute this group, we can view the exterior $E(S)$ itself as a compact $4$-manifold with boundary. Its handle decomposition (and thus its Kirby diagram) can be derived from the diagram of $(X, S)$ by treating the meridians of the surface $S$ as new $1$-handles. We first state the general proposition for determining the fundamental group of any connected $4$-manifold $X$ represented by a Kirby diagram $\mathcal{K}$. We will then apply this result to our specific case of the surface exterior.

\begin{figure}[H]
    \centering
    \includegraphics[width=0.3\linewidth]{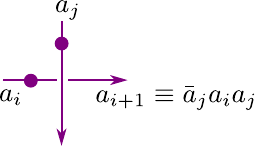}
    \caption{1-handle crossing relation}
    \label{fig:Kirby_diag_relation}
\end{figure}

\begin{prop}\label{prop:Kirby_fnm_group}
    Let $\mathcal{K}$ be a Kirby diagram representing a connected $4$-manifold $X$. 
    Let $a_1, \dots, a_n$ denote the set of generators corresponding to all $1$-handles 
    of $X$ (represented by dotted circles). 
    For each component of the $2$-handle attaching link, let $w_j$ be the word in the 
    free group generated by $\{a_1, \dots, a_n\}$, obtained by reading the sequence of 
    over-arcs along the attaching circle. 
    Then the fundamental group $\pi_1(X)$ admits a presentation:
    \[
        \pi_1(X) = \langle a_1, \dots, a_n \mid r_1, \dots, r_k, \; w_1, \dots, w_m \rangle,
    \]
    where $r_1, \dots, r_k$ are the relations encoded by the crossings of the dotted 
    link itself, as illustrated in Figure~\ref{fig:Kirby_diag_relation}.
\end{prop}

\begin{proof}
    The proof follows from the Seifert--van Kampen theorem applied to the handle decomposition of the $4$-manifold $X$.
    
    Let $X_1$ be the union of the $0$-handle and the $1$-handles. Topologically, $X_1$ is homotopy equivalent to the wedge sum of circles, $\bigvee^n S^1$.
    In the context of Kirby diagrams, $X_1$ can be viewed as the complement of the dotted link $L_1$ in the $3$-sphere (which is the boundary of the $0$-handle).
    The fundamental group $\pi_1(X_1)$ is thus the link group of $L_1$. 
    The generators $a_1, \dots, a_n$ correspond to the meridians of the components of $L_1$ (which geometrically represent the cores of the $1$-handles). 
    The relations $r_1, \dots, r_k$ arise from the Wirtinger presentation associated with the crossings of the dotted link diagram \cite{rolfsen1990knots}. Thus, we have the presentation:
    \[
        \pi_1(X_1) = \langle a_1, \dots, a_n \mid r_1, \dots, r_k \rangle.
    \]

    Next, let $X_2$ be the manifold obtained by attaching $2$-handles to $X_1$. Each $2$-handle is attached along a framed knot $L_{2,j}$ in $\partial X_1$. Attaching the $2$-handles introduces relations $w_1, \dots, w_m$ via the Seifert--van Kampen theorem, and the subsequent attachment of $3$- and $4$-handles does not affect the fundamental 
    group. Hence, $\pi_1(X) \cong \pi_1(X_2)$, concluding the proof.
\end{proof}

\begin{example}
    Let $C_d$ denote the degree-$d$ algebraic curve, which is the smooth isotopy class of the zero-set of the homogeneous polynomial $z_0^d+z_1^d+z_2^d$ in $\CP^2$. We can draw a banded unlink diagram of $C_d$ by using the method described in \cite{gompf19994}.

    Consider an $(m,n)$-torus link $t_{m,n}$, which is an oriented link in $S^3$ consisting $\gcd(m,n)$ components, and lies on the boundary $T^2$ of a tubular neighborhood of an unknot, representing the class $m\mu+n\lambda\in H_1(T^2)$. Let $F_{m,n}$ be the Seifert surface of $t_{m,n}$. The surface $F_{m,n}$ is a oriented surface constructed by connecting $n$ disks with $m(n-1)$ bands. 
    We could obtain a singular degree-$d$ curve in $\CP^2$ as union of $d$ generic complex lines. By applying a perturbation we could resolve the transverse double point by removing a pair of intersecting disks near each singular point, and replace them by an annulus. We observe that this annulus is the Seifert surface $F_{2,2}$ of $t_{2,2}$, lying in $\partial D^4$. 

    Now we draw a degree-$d$ curve in $\CP^2$. Let $F^*$ be a singular curve consisting of $d$ complex lines. The curve $F^*$ intersects the boundary of the $2$-handle of $\CP^2$ in a $(d,d)$-torus link $T_{d,d}$. Consider the level set of a Morse function on $\CP^2$. As the level decreases from the $2$-handle into $I\times D^3$, the link $t_{d,d}$ separates into an unlink bounding $d$ disks. Each transverse double point of the singular curve $F^*$, corresponding to the intersection of a pair of complex lines, gives rise to a positive crossing between two components of the link. Resolving these double points corresponds to attaching bands between the disks, producing a smooth surface $F$. Since the resolutions can be arranged to occur at the same level, the resulting surface is a Seifert surface for $t_{d,d}$. The remaining part of $F$ lies in the $2$-handle and consists of $d$ parallel copies of its core. An example for $d=3$ is shown in Figure~\ref{fig:cpx_curve}.

    Label the generators corresponding to the dotted link by $x_1, \dots, x_d$. Each $2$-handle associated with a band induces the relation $x_i x_{i+1}^{-1}=1$, which implies $x_i = x_{i+1}$ for all $i$. Let us denote this common generator by $x$. The $(+1)$-framed $2$-handle imposes the relation $x_1 x_2 \dots x_d = 1$, which simplifies to $x^d = 1$. Thus, the fundamental group of $E(C_d)$ admits the presentation
    \[
        \pi_1(E(C_d)) \cong \langle x \mid x^d = 1 \rangle,
    \]
    which is isomorphic to the cyclic group $\mathbb{Z}_d$.
\end{example}

    \begin{figure}[H]
    \centering
    \includegraphics[width=0.8\linewidth]{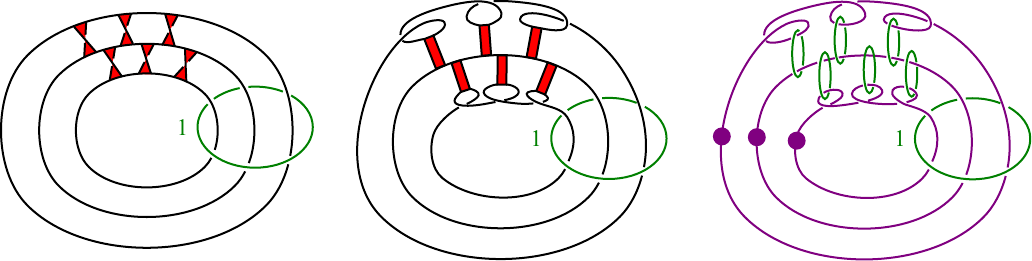}
    \caption{Left: A banded unlink diagram of the degree-$3$ algebraic curve $C_3$. 
    Middle: The diagram with bands arranged in the blackboard framing. 
    Right: The associated Kirby diagram of the exterior $E(C_3)$, where the framing of each $2$-handle is $0$.}
    \label{fig:cpx_curve}
	\end{figure}


\subsection{A Wirtinger type presentation of the surface knot quandle}

For classical knot in $3$-manifold, the fundamental quandle can be represented by the knot diagram, for the precise presentation, see~\cite{fenn1992racks}. 

\begin{thm}[\cite{fenn1992racks}]\label{thm:3-mfd_quandle}
    Let $l$ be a link embedded in a closed orientable $3$-manifold $M$. The fundamental augmented quandle $Q(l)$ admits a presentation derived from any diagram of $l$, where the generators correspond precisely to the arcs of the diagram, and the relations correspond to the Wirtinger crossing relations.
\end{thm}

Let $\D=(\K, L, v)$ be a banded unlink diagram of $S$ in $X$. Give an orientation of the diagram, we define the diagrammatic quandle $Q(\D)$ as bellow. 

\begin{itemize}
	\item Let $\mathbf{x}=\{x_{1},\dots,x_{n}\}$ be the labels of arcs of $L$, serving as the primary generators. 
    
    \item Let $\mathbf{a}=\{a_{1},\dots,a_{m}\}$ be the labels of the arcs of the dotted link $L_1$, serving as the operator generators. The elements $a_{1},\dots,a_{m}$ have relations coming from $\pi_1(X)$. 

	\item At each crossing where the under arcs belong to $L$, label the under arcs by $x_i$ and $x_{i+1}$. 
    When the over arc is $x_j$ in $\mathbf{x}$ or $a_j$ in $\mathbf{a}$, such crossings give rise to primary relations of the form $x_{i+1}=x_i^{x_j}$ or $x_{i+1}=x_i^{a_j}$, respectively. 
    When the over-strand is a band in $v$ or an arc in $L_2$, it does not change the associated element. Denote the primary relations by $p_1,\dots p_t$. 

    \begin{figure}[H]
        \centering
        \includegraphics[width=0.9\linewidth]{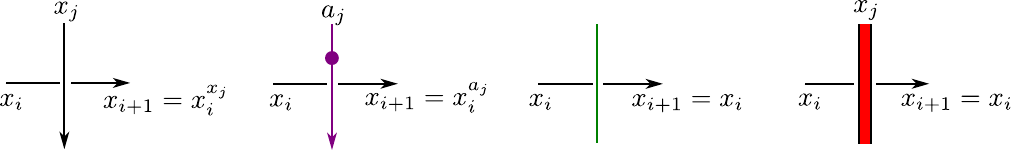}
        \caption{Primary relations}
        \label{fig:primary_relations}
    \end{figure}

    At each crossing where the under arcs belongs to $L_1$, label the under-strand by $a_i$ and $a_{i+1}$. When the over arc is $x_j$ in $\mathbf{x}$ or $a_j$ in $\mathbf{a}$, such crossings give rise to operator relations of the form $a_{i+1} \equiv \bar{x}_j a_i x_j$ or $a_{i+1} \equiv \bar{a}_j a_i a_j$, respectively. When the over-strand is a band in $v$ or an arc in $L_2$, it does not change the associated element, and the relation reduces to $a_{i+1} = a_i$. We denote the operator relations by $o_1, \dots, o_k$.

    \begin{figure}[H]
        \centering
        \includegraphics[width=0.9\linewidth]{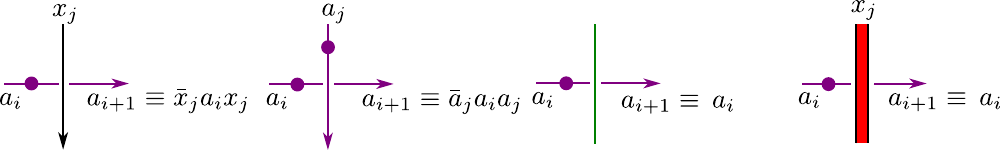}
        \caption{Operator relations.}
        \label{fig:oprator_relations}
    \end{figure}

    \item Additional primary relations are given by the bands $b_i$ in $v$. Each band $b_i$ gives rise to a primary relation $x_j^w=x_k$ between the two arcs $x_j$ and $x_k$ connected by $b_i$, where $w$ is a word obtained by reading the labels encountered along $b_i$ from $x_j$ to $x_k$. If $b_i$ does not pass under any arcs of $\mathbf{x}$ or $\mathbf{a}$, we declare that $w\equiv 1$. We denote these relations by $b_1,\dots,b_s$ as well.
    
    \begin{figure}[H]
        \centering
        \includegraphics[width=0.4\linewidth]{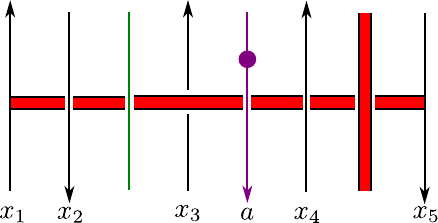}
        \caption{An example of a band relation. Reading the word along the band from $x_1$ to $x_5$ yields the primary relation $x_5 = x_1^{x_2 a \bar{x}_4}$.}.
        \label{fig:band_relation}
    \end{figure}

	\item Additional operator relations $s_1, \dots, s_l$ are obtained from the $2$-handle attaching link $L_2$. For each component of $L_2$, we read the word in $\mathbf{x} \cup \mathbf{a}$ formed by the arcs passing over that component, traversing it once in the positive direction. Setting this word equal to the identity yields an operator relation.
        \begin{figure}[H]
        \centering
        \includegraphics[width=0.22\linewidth]{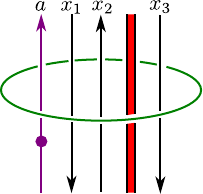}
        \caption{An example of a $2$-handle relation. Reading the arcs passing over the $2$-handle attaching link yields the operator relation $\bar x_3 x_2 \bar x_1 a \equiv 1$.}
        \label{fig:2-handle}
    \end{figure}
	\end{itemize}

\begin{defn}
The diagrammatic quandle $Q(\mathcal{D})$ associated with $\D$ is defined as a quotient of the extended free quandle $FQ(S_P, S_O)$. Specifically, $Q(\mathcal{D})$ is given by the presentation
\[
Q(\mathcal{D}) = [ S_P , S_O \mid R_P , R_O ] = S_P \times F(S_P \cup S_O)/ \sim,
\]
where the sets of generators and relations are defined as follows:
\begin{itemize}
    \item $S_P = \mathbf{x}$ is the set of primary generators;
    \item $S_O = \mathbf{a}$ is the set of operator generators;
    \item $R_P$ is the set of primary relations, consisting of the crossing relations $\{p_1, \dots, p_t\}$ and the band relations $\{b_1, \dots, b_s\}$;
    \item $R_O$ is the set of operator relations, consisting of the relations $\{o_1, \dots, o_k\}$ derived from the crossings, and the relations $\{s_1, \dots, s_l\}$ derived from the $2$-handle attaching link $L_2$.
\end{itemize}
\end{defn}

\begin{remark}\label{rmk:augmented_map}
    The geometric meaning of this presentation is justified by the augmentation map $\partial$. For any element of the form $x_i^a$, where $x_i$ in $\mathbf{x}$ and $a$ is a word in the free group $F(\mathbf{x} \cup \mathbf{a})$, applying the augmentation map yields the identity:
    \[
        \partial(x_i^a) = \partial(a)^{-1} \partial(x_i) \partial(a).
    \]
\end{remark}

We can derive the operator group from $Q(\mathcal{D})$ by applying the augmentation map $\partial$. 
The generator of this operator group is $\mathbf{x}\cup \mathbf{a}$. 
\begin{itemize}
    \item The set of primary generators $\mathbf{x}$ corresponds to the surface link components. Recall that the exterior $E(S)$ is obtained by removing the open tubular neighborhood of the surface $S$ from $X$. Topologically, removing the neighborhood of a surface $0$-handle is equivalent to attaching a $1$-handle to the complement. Consequently, each arc in $\mathbf{x}$ in the banded unlink diagram $\mathcal{D}$ is converted into a dotted circle in $\mathcal{K}$.  
    \item The set of operator generators $\mathbf{a}$ corresponds to the $1$-handles of the ambient four-manifold $X$, which are represented by dotted circles in the original diagram of $X$. 
\end{itemize}

Operator relations in $R_O$ remain unchanged under the augmentation map, and every primary relation $p_i$ and $b_i$ in $R_P$ induces a group relation by replacing quandle operations with conjugations. We denote the resulting set of group relations by $R'_P$, which consists of two types:
\begin{itemize}
    \item \textit{Primary crossing relations} $p'_i$: corresponding to the crossing relation $x_j = x_i^y$ in $Q(\mathcal{D})$, taking the form $x_j = \bar{y}x_iy$;
    \item \textit{Band relations} $b'_k$: corresponding to the band relation $x_j = x_i^w$, taking the form $x_j = \bar{w}x_iw$.
\end{itemize}

The following lemma verifies that the group presented by taking the augmentation is indeed the knot group of the surface link.

\begin{lem}\label{lem:knot_group_BUD}
    Let $S$ be a surface link in $X$ and $\mathcal{D}$ be a banded unlink diagram of $S$. 
    Let $G(\mathcal{D})$ be the group defined by the presentation induced from $Q(\mathcal{D})$:
    \[
        G(\mathcal{D}) = \langle \mathbf{x}\cup \mathbf{a} \mid R'_P, R_O \rangle.
    \]
    Then, this group $G(\mathcal{D})$ is isomorphic to the knot group $G(S)$ of $S$.
\end{lem}

\begin{proof}
    Let $\mathcal{K}$ be the Kirby diagram of the exterior $E(S)$ constructed from $\mathcal{D}$, and let $G(\mathcal{K})$ denote the presentation of $G(S)$ obtained by Proposition~\ref{prop:Kirby_fnm_group}. It suffices to show that $G(\mathcal{K}) \cong G(\mathcal{D})$.
    
    By the construction above, there is a bijection between generators of $G(\D)$ and generators of $G(\K)$. Crossing relations $p'_k$ in $R'_{P}$ are identical to the Wirtinger relations $r_l$ arising from the crossings of the dotted link components. For the band relations $b'_k$, let $b$ be a band connecting $x_i$ to $x_j$, and let $u$ be the word obtained by reading the over-passing arcs along $b$. The word associated with the corresponding $2$-handle is $w_m = ux_j\bar{u}\bar{x}_i$, and the relation $w_m = 1$ gives $x_j = \bar{u}x_iu$, which is exactly the band relation $b'_k$. The operator relations $R_O$ correspond to the Wirtinger and $2$-handle relations of $\mathcal{K}$. Therefore $G(\mathcal{K}) \cong G(\mathcal{D})$.

\end{proof}

\begin{remark}\label{rmk:operator_group}
    It is important to note that the group $G(\mathcal{D})$ constructed in Lemma~\ref{lem:knot_group_BUD} is precisely the operator group associated with the augmented presentation of $Q(\mathcal{D})$.  By definition, the operator group of an augmented quandle presented by generators $S_P \cup S_O$ and relations $R_P \cup R_O$ is the group defined by the same generators and the induced group relations (where quandle operations are replaced by conjugations). Thus, Lemma~\ref{lem:knot_group_BUD} establishes that the abstract operator group of the fundamental augmented quandle is isomorphic to the knot group $G(S)$.
\end{remark}

\begin{thm}\label{thm:qdle_presentation}
	Let $S$ be a surface link and $\D$ be a banded unlink diagram of $\K$. Then $Q(S)$ and $Q(\D)$ are isomorphic. 
\end{thm}

\begin{proof}

    We shall define two quandle homomorphisms $\mu:Q(\D)\to Q(S)$ and $\lambda:Q(S)\to Q(\D)$, such that $\mu\circ\lambda=\lambda\circ\mu=id$. 

    Let us define $\mu: Q(\mathcal{D}) \to Q(S)$.

    Let $*$ be the base point in $X_{3/2} \setminus S$. For each primary generator $x_i$ in $\mathbf{x}$ of $Q(\mathcal{D})$, define $\mu(x_i)$ to be the homotopy class $[N_i]$ of a noose $N_i=(D_i, \xi_i)$ in $X_{3/2}$, such that the disk $D_i$ is a meridian disk of the arc represented by $x_i$, and the path $\xi_i$ starts at a point on $\partial D_i$ and ends at $*$. We choose $\xi_i$ so that, when drawn on $\mathcal{D}$, it passes over all other arcs of $\mathcal{D}$.
    
    By Lemma~\ref{lem:knot_group_BUD}, we identify the operator group of $Q(\mathcal{D})$ with the knot group $G(S)$. 
    Consequently, every word $b$ in the free group $F(\mathbf{x} \cup \mathbf{a})$ determines a well-defined element in the knot group $G(S)$. Let $\gamma_b$ be a loop embedded in $E(S)$ representing the group element determined by $b$. We define the action of $\mu$ by:
    \[
        \mu(x_i^{b}) = [N_i]^{[\gamma_b]} = [(D_i, \xi_i \cdot \gamma_b)].
    \]
    
    We verify the well-definedness by checking the relations. 

    \textit{Crossing relations} ($p_i$):
    Consider a crossing where an arc $y$ passes over an arc $x$, dividing it into $x_j$ and $x_k$. The relation is $x_k = x_j^y$. 
    In $Q(S)$, sliding the meridian disk $D_j$ along the under-passing arc $x$ to the position of $D_k$ requires passing under the arc $y$. This corresponds to conjugating the noose by the meridian of $y$. Thus, $\mu(x_k) = \mu(x_j)^{\mu(y)}$, satisfying the relation.

    \textit{Band relations} ($b_i$):
    Let $x_j$ and $x_k$ be arcs connected by a band $b$, with the associated word $w$ obtained by reading the arcs passing over the band. The relation is $x_k = x_j^w$. 
    Let $\gamma_w$ be a loop in $E(S)$ representing the group element determined by $w$. 
    Geometrically, the meridian disk $D_j$ can be isotoped along the band $b$ to become $D_k$. During this isotopy, the path $\xi_j$ is extended by the path running along the band, which traces exactly the loop $\gamma_w$.
    Thus, 
    \[
        \mu(x_k) = [(D_k, \xi_k)] = [(D_j, \xi_j \cdot \gamma_w)] = \mu(x_j)^w.
    \]

    \textit{Operator relations} ($R_O$):
    The relations in $R_O$ define the structure of the operator group. Since we have identified the operator group of $Q(\mathcal{D})$ with $G(S)$, any relation $u = v$ holding in the operator group implies that the words $u$ and $v$ represent the identical element in the knot group $G(S)$. 
    Consequently, their actions on the quandle element are identical: $\mu(x^u) = \mu(x^v)$.

   Then, let us define $\lambda:Q(S)\to Q(\D)$.  
   
   Let $S_+=S\cap f^{-1}[3/2,4]$ and $S_-=S\cap f^{-1}[0,3/2]$. Then $S_+\cap S_-$ is the banded unlink $S_0$ in $X_{3/2}$. Fix a basepoint $*$ in $X_{3/2}\setminus S$. 
    
    First, for nooses on $S_{\pm}$, we claim that:
    
    \begin{claim}\label{claim:N}
        For every element $[N]$ in $Q(S)$, where $N$ is a noose on $S_+$, there exist a noose $N_0=(D_0,\xi_0)$ in $ X_{3/2}$, and a loop $\alpha$ in $E(S)$ based at $*$, such that $[N]=[N_0]^{[\alpha]}$. 
    \end{claim}
    
    \begin{proof}[Proof of Claim \ref{claim:N}]
        Let $\Delta$ be a trivial disk in $S_+$. 
        Choose one point $p$ on $\Delta$, take the meridian disk $D$ at $p$, and take a path $\xi$ in $E(S)$ from a point on $\partial D$ to $*$. Then $N=(D,\xi)$ is a noose on $\Delta$. 
        Choose a noose $N_0=(D_0,\xi_0)$ in $ X_{3/2}$ such that $D_0\cap S_0$ is a point of the arc $x_0$ in the diagram $\D$, and $\xi_i$ is a path passing over all arcs when drawn on $\D$.

        Choose paths as follows:
        \begin{itemize}
          \item a path $\rho_1$ in $D_0$ from the start point of $\xi_0$ to $p_0$;
          \item a path $\rho_2$ in $\Delta$ from $p_0$ to $p$;
          \item a path $\rho_3$ in $D$ from $p$ to the start point of $\xi$.
        \end{itemize}
        
        By pushing $\rho_2$ slightly in the normal direction of $\Delta$, we may also assume that $\rho_2$ lies in $E(S)$. Then $\rho=\rho_1\cdot \rho_2\cdot\rho_3$ is a path from the start point of $\xi_0$ to the start point of $\xi$, contained in $E(S)$. Define a loop $\alpha=\xi_0^{-1}\cdot \rho \cdot \xi$ in $E(S)$. 
        
        Since $[\alpha]\in G(S)$ is an operator of $Q(S)$, we have $[N]=[N_0]^{[\alpha]}=[(D_0,\xi_0\cdot\alpha)]$. 

        Nooses on $S_-$ can be shown in the same way. 
    \end{proof}    
    
    Let us continue the proof of Theorem~\ref{thm:qdle_presentation}. 
    
    Take a noose $N$ on $S_+$ or $S_-$, then $[N]$ is an element in $Q(S)$. By the above claim, write $[N]=[N_0]^{[\alpha]}$.
    Let $w$ be a word representing $[\alpha]\in G(S)$ using the presentation
    described in Lemma~\ref{lem:knot_group_BUD}.
    We define
    $$\lambda([N]) :=x_0^w.$$
    where $x_0$ in $Q(\D)$ is determined by $N_0$. 
    
    Finally, we verify that $\lambda: Q(S) \to Q(\mathcal{D})$ is well-defined. Suppose an element $[N]$ of $Q(S)$ is represented by two pairs $(N_0, [\gamma])$ and $(N'_0, [\gamma'])$. Let $y = \lambda([N_0])^{[\gamma]}$ and $y' = \lambda([N'_0])^{[\gamma']}$ be their respective images in $Q(\mathcal{D})$. We must show that $y = y'$.

    Since the representing nooses are homotopic, this homotopy traces out an immersed disk $H$ in $E(S)$ bounding $\gamma$ and $\gamma'$, and a path on the boundary of the tubular neighborhood of $S$. We can regard $H$ as an element of the relative homotopy group $\pi_2(E(S), E(S) \cap X_{3/2})$. Since $E(S) \cap X_{3/2}$ is the boundary of the union of $0$-handles and $1$-handles of $E(S)$, this relative homotopy group is spanned by the core disks of the $2$-handles of $E(S)$. Then $H$ can be homotoped into the union of $X_{3/2}\cap E(S)$ and the core disks of $2$-handles, and can be decomposed in the relative homotopy group as a product $H = H_1H_2\cdots H_m$, where $H_i$ is a homotopy of a local move, either whose image is entirely contained in $X_{3/2}\cap E(S)$, or represents a homotopy across the core disk of a $2$-handle.

    When $H_i$ is a homotopy contained in $X_{3/2}\cap E(S)$, it can be represented by the Wirtinger presentation of link diagram by Theorem~\ref{thm:3-mfd_quandle}. 
    When $H_i$ is a homotopy across the core disk of $2$-handle of $E(S)$, there is a relation in $G(S)$ corresponding to this $2$-handle attaching link. 
    
    By using the decomposition of $H$, the deformation of the path $\gamma$ can be decomposed into a finite sequence of intermediate paths, say $\gamma_0, \gamma_1, \dots, \gamma_m$, where $\gamma_0$ represents $y$ and $\gamma_m$ represents $y'$, and $H_i$ is a homotopy from $\gamma_{i-1}$ to $\gamma_i$. For each step, if the homotopy $H_i$ is entirely contained in $X_{3/2}\cap E(S)$, it corresponds to a local move on the diagram, and the algebraic expressions for $\gamma_{i-1}$ and $\gamma_i$ differ by exactly one application of a Wirtinger relation. On the other hand, if $H_i$ crosses the core disk of a $2$-handle, their algebraic expressions differ by a conjugate of a band relation $b_j$ or a $2$-handle relation $s_j$. However, since $Q(\mathcal{D})$ is defined as the quotient space modulo exactly these relations, the equivalence classes of these successive nooses are forced to be equal, yielding $[\gamma_{i-1}] = [\gamma_i]$. By transitivity along the sequence of local moves, we obtain the equalities $y=y'$. This completes the proof.

    By the definition, $\mu$ and $\lambda$ are both quandle homomorphism, and $\mu\circ\lambda=id_{Q_({S})}$ and $\lambda\circ\mu=id_{Q_{(\D)}}$. 
\end{proof}

\begin{remark}
    In the proof of Theorem~\ref{thm:qdle_presentation}, the diagram $\D$ of $S$ is chosen arbitrarily. Therefore, the resulting quandle presentation $Q(\D)$ is independent of the choice of $\D$.
\end{remark}

Let $X$ be a $4$-manifold. A surface link $S$ embedded in $X$ is said to be \emph{local} if there exists a smoothly embedded $4$-ball $B^4 \subset X$ such that $S$ is entirely contained in the interior of $B^4$.

\begin{cor}\label{cor:local_surface}
    Let $X$ be a closed $4$-manifold admits a handle decomposition without $1$-handles. Then the fundamental quandle of a local surface link $S$ is isomorphic to that of the corresponding surface link in $S^4$.
\end{cor}

\begin{proof}
    We can find a banded unlink diagram $\D$ of $(X,S)$ where $L$ and $v$ are separated from, and thus have no crossings with $\K$. Thus there are only primary generators and relations in the presentation of $Q(S)$, which is equal to the presentation of the fundamental quandle of the corresponding surface link in $S^4$.
\end{proof}

\begin{example}[Degree $d$ Algebraic curve $C_d$ in $\CP^2$]

    A banded unlink diagram of $C_d$ is shown as Figure~\ref{fig:cpx_curve}. From this diagram, we see that there is a unique primary generator, namely $x$, and no operator generators.
    Reading around the attaching circle of the $2$-handle yields the operator relation $x^d \equiv 1$. It follows that $Q(C_d)=[x\mid x^d\equiv1]=[x]$ for every $d$, that is, the quandle generated by a single element.
	
	For surface knot in $S^4$, the associated group of the knot quandle agrees with the knot group. However, this correspondence does not persist in other ambient 4-manifolds. Using Proposition~\ref{prop:Kirby_fnm_group} together with the broken surface diagrams of algebraic curves, we compute that the fundamental group of the complement of $C_d$ in $\CP^2$ is isomorphic to $\Z_d$. 
    \end{example}

\begin{lem}\label{lem:=1}
	Let $\D$ be a banded unlink diagram of $\K$, and let $C=\{x_1,\dots ,x_n\}$ be the generators corresponding to the arcs contained in a single component of $\D$. If $x_i^d\equiv 1$ for some $x_i\in C$, then $x^d_i \equiv 1$ for every $x_i$ in $C$. 
\end{lem}

\begin{proof}
	Consider the crossing $c$ which include $x_i$ in $C$. The primary relation of the crossing can be written as $x_i*y^\epsilon=x_{i+1}$, where $y$ is the over arc of $c$, and $\epsilon=\pm1$ depends on the orientation of $y$. It follows that $x_{i+1}^d\equiv (\bar y^\epsilon x_iy^\epsilon)^d\equiv1$. By iterating this argument along every crossings around the component $C$, we conclude that $x_j^d\equiv 1$ for every $x_j$ in $ C$.
\end{proof}

\begin{cor}\label{cor:no_1-handle}
    Let $X$ be a $4$-manifold without $1$-handle. Then the fundamental quandle of every surface link $S$ in $X$ has primary presentation.
\end{cor}

\begin{proof}
    The banded unlink diagram $\D$ of $S$ has no $1$-handle, so $Q(\D)$ has no operator generator, and every operator relation can be written as $w\equiv1$ for some word $w$. Rewrite this relation into $x_i^w=x_i$ for every primary generator $x_i$, every relation is a primary relation. 
\end{proof}

\begin{example}
    Let $S_d$ be the connected sum of the spun trefoil and the degree-$d$ algebraic curve $C_d$ in $\CP^2$. 
    The banded unlink diagram of $S_1$ is shown in the left-hand picture of Figure~\ref{fig:BUD->Kirby}.
    
    The fundamental quandle $Q(S_d)$ admits a presentation
    \[
    [x_1,x_2 \mid x_1^{x_2}=x_2^{\bar x_1},\ x_1^d\equiv x_2^d\equiv1].
    \]
\end{example}

We now generalize the previous example as a corollary of Theorem~\ref{thm:qdle_presentation}.

\begin{cor}\label{cor:K(S+C_d)}
    Let $C_d$ be a degree-$d$ algebraic curve in $\CP^2$, and let $S$ be a surface knot embedded in $S^4$. Then the fundamental kei $K(S\#C_d)$ is determined by $d$ and $K(S)$: 
    \begin{itemize}
        \item If $d=2k$ for some integer $k$, then $K(S\#C_d)$ is isomorphic to $K(S)$.
        \item If $d=2k-1$ for some integer $k$, then $K(S\#C_d)$ is isomorphic to $\langle x\rangle$, which is the kei generated by a single element.
    \end{itemize}
\end{cor}

\begin{proof}
    We prove this corollary by computing the diagram. The fundamental kei $K(S)$ of surface knot $S$ in $S^4$ only has primary generators and relations. By taking the connected sum with $C_d$, no new generators are introduced, but new operator relations $x_i^d\equiv1$ are added to the presentation. 
    If $d=2k$, then the relation $x_i^{2k}\equiv1$ follows from the second axiom of a kei, and hence $K(S\#C_d)\cong K(S)$.  On the other hand, If $d=2k-1$, combining the second axiom of kei with the relation $x_i^{2k-1}\equiv1$, we get $x_i\equiv1$. This implies $K(S\#C_d)\cong \langle x\rangle$.
\end{proof}


\subsection{Quandle colorings}

We will study the set of homomorphisms $\Hom(Q(\mathcal{D}), Q)$.

\begin{defn}
Let $Q$ be a finite quandle, and let $S$ be a surface link embedded in a $4$-manifold $X$. A \textit{$Q$-coloring} of $S$ is a quandle homomorphism $c: Q(S) \to Q$. We denote by $\Col_Q(S)$ the set of all $Q$-colorings of $S$. The cardinality $\#\Col_Q(S)$ is called the \textit{$Q$-coloring number} of $S$.
\end{defn}

By Theorem~\ref{thm:qdle_presentation}, the coloring number is an invariant of the surface link $S$.

\begin{example}
	For any integer $d$, the degree-$d$ algebrain curve $C_d$ admits only the constant coloring. Consequently, $\#\Col_Q(C_d)=\#Q$ for any finite quandle $Q$. 
\end{example}

\begin{example}
    Let $S_2$ be the connected sum of the spun trefoil and the degree-$2$ algebraic curve. By computation, the $R_3$-coloring number $\#\Col_{R_3}(S_2)=9$. 
\end{example}


\section{Applications to bridge number of surface knots in $4$-manifolds}\label{sec:applications}

\subsection{Trisections and surface links}

We first recall the definition of trisection and bridge trisection, following \cite{gay2016trisecting} and \cite{meier2018bridge}. In particular, bridge trisections provide a framework that generalizes the notion of bridge number for classical knots to surfaces embedded in $4$-manifolds.

\begin{defn}
	Let $X$ be a closed $4$-manifold. A $(g, k_1, k_2, k_3)$\textit{-trisection} $\mathcal{T}$ of $X$ is a decomposition $X=X_1\cup X_2\cup X_3$ where 
	\begin{itemize}
		\item $X_i\cong \#^{k_i} S^1\times D^3$ for each $i$ in $ \{1,2,3\}$; 
		\item $X_{ij}=X_i\cap X_j=\partial X_i \cap\partial X_j\cong \natural^g S^1\times D^2$ for each $i,j$ in $ \{1,2,3\}$ and $i\neq j$; 
		\item $\Sigma=X_1\cap X_2\cap X_3\cong \#^g S^1\times S^1$.
	\end{itemize}
\end{defn}

\begin{defn}
	Let $X$ be a $4$-manifold with a trisection $\mathcal{T}$, and let $S$ be a closed surface embedded in $X$. We say $S$ is in $(g,k_1,k_2,k_3; b,c_1,c_2,c_3)$-\textit{bridge position} if 
	\begin{itemize}
		\item $D_i=S\cap X_i$ is a collection of $c_i$ embedded boundary parallel disjoint disks in $X_i$;
		\item $\tau_{ij}=S\cap X_{ij}$ is a collection of $b$ boundary parallel disjoint arcs in $X_{ij}$; 
		\item $S\cap \Sigma$ consists of $2b$ points. 
	\end{itemize}
	If $S$ is in the bridge position in $X$, then we call the decomposition $(X,D)=(X_1,D_1)\cup (X_2,D_2)\cup (X_3,D_3)$ a $(g,k_1,k_2,k_3; b,c_1,c_2,c_3)$-\textit{bridge trisection}. The \textit{bridge number} of $S$, denoted by $b(S)$, is defined to be the minimum value of $b$ among all bridge trisections of $(X,S)$.
\end{defn}

A bridge trisection of $S$ can be constructed from a banded unlink diagram of $S$, associated with a Morse function $h$. The disks corresponding to the maxima and minima of $S$ are contained in $X_1$ and $X_3$, respectively, while neighborhoods of the bands are contained in $X_2$. More precisely, see \cite{meier2018bridge}. 

\begin{prop}[\cite{meier2018bridge}]\label{prop:bridge_trisection}
	Let $S$ in $X$ be a surface embedded in $4$-manifold $X$, with a $(g,k_1,k_2,k_3; b,c_1,c_2,c_3)$-bridge trisection $\mathcal{T}$. Then there exists a Morse function $h$ for $(X,S)$ such that 
	\begin{itemize}
		\item $h$ has $k_1$ index-$1$ critical points, $g-k_2$ index-$2$ critical points, and $k_3$ index-$3$ critical points; 
		\item $h|_S$ has $c_1$ minima, $b-c_2$ saddles, and $c_3$ maxima. 
	\end{itemize}
\end{prop}

\subsection{Surface knots in $4$-manifolds with the common bridge numbers}

In this section, we study the bridge numbers of surface knots. Sato and Tanaka \cite{sato2022bridge} proved that for any integer $b \geq 4$, there exist infinitely many surface knots with bridge number $b$ in $S^4$. As a direct consequence of Theorem~\ref{thm:qdle_presentation} and Corollary~\ref{cor:local_surface}, this result extends to local surface knots in arbitrary $4$-manifolds without $1$-handles:

\begin{cor}\label{cor:local_infinite_many}
    Let $X$ be a closed $4$-manifold which admits a handle decomposition without $1$-handle. For any integer $b\geq4$, there exist infinitely many pairwise non-isotopic local embedded surfaces with bridge number $b$ in $X$. 
\end{cor}

\begin{proof}
    Take the $4$-ball $B^4$ which contain the surface link $S$ in $X$. Embed this $B^4$ in $S^4$, we denote the surface in $S^4$ by $S'$. By Corollary~\ref{cor:local_surface}, $Q(S)$ and $Q(S')$ are isomorphic, so $\#\Col_K(S)=\#\Col_K(S')$ for any finite kei $K$. Let $S\subset B^4$ be the $0$-twist-spun knot of the connected sum of $(2,q)$-torus knot $\#^{l}t_{2,q}$, the result can be proved by the same discussion as \cite{sato2022bridge}. 
\end{proof}

We further construct explicit examples of surfaces representing non-trivial homology classes in $\CP^2\#m\overline{\CP}^2$, which are in particular not local, by employing a lower bound for the bridge number which depends on the kei coloring number:

\begin{thm}\label{thm:infinitely_many_bridge}
	For any integer $b\geq4$ and $m\geq0$, there exist infinitely many pairwise homologous non-isotopic embedded surfaces with bridge number $b$ in $\CP^2\#m\overline{\CP}^2$.
\end{thm}

As a generalization of the inequality established by Sato and Tanaka for surface knots in $S^4$ \cite{sato2022bridge}, we establish the following lower bound for the bridge number of surface knots in arbitrary $4$-manifolds.

\begin{prop}\label{inequality}
	Let $S$ be a surface knot in $X^4$. For any finite quandle $Q$, the following inequality holds:
	$$ b(S)\geq 3\log_{\#Q}(\#\Col_{Q}(S))-\chi(S)$$
\end{prop}

\begin{proof}
    By Proposition~\ref{prop:bridge_trisection}, there exists a Morse function with $c_1$ index-$0$ critical points, $b(S)-c_2$ index-$1$ critical points, and $c_3$ index-$2$ critical points, which induces a banded unlink diagram $\D$ of $S$. 
    Consider the $Q$-colorings $\Col_Q(\D)=\Col_Q(S)$. The inclusions of the minimal and maximal disk systems into the banded unlink diagram $\D$ induce injective maps on $Q$-colorings:
    $$\phi_{min}:\Col_Q(S)\to \Col_Q(L_{c_1})\quad \phi_{max}:\Col_Q(S)\to \Col_Q(L_{c_3}),$$ 
    where $L_{c_i}$ is the $c_i$-component classical unlink. It follows that $\#Q^{\min\{c_1,c_3\}}\geq \#\Col_Q(S)$. By permuting the roles of the three sectors in the bridge trisection, the same argument applies to each pair of indices, and hence $\#Q^{\min\{c_1,c_2,c_3\}}\geq \#\Col_Q(S)$. 
    
    On the other hand, considering the cell decomposition induced by $h$, we have $\chi(S)=c_1+c_2+c_3-b(S)$. It follows that 
    $$b(S) = c_1+c_2+c_3-\chi(S) \geq 3\min\{c_1,c_2,c_3\}-\chi(S) \geq 3\log_{\#Q}(\#\Col_Q(S))-\chi(S).$$
\end{proof}

We construct a surface $\Sigma_d$ in $\CP^2\#m\overline{\CP}^2$ as follows for any integer $p\geq 0$. Take a degree-$d$ algebraic curve $C_d$ in $\CP^2$ and its orientation-reversed copy $\overline{C_d}$ in $\overline{\CP}^2$. When forming the connected sum $\CP^2\#\overline{\CP}^2$, we simultaneously performing a band sum between $C_d$ and $\overline{C}_d$. Denote the resulting surface $C_d\#m\overline{C}_d$  by $\Sigma_d$. The banded unlink diagram of $\Sigma_2$ is shown as Figure~\ref{fig:Sigma_d} bellow. 

\begin{figure}[H]
    \centering
    \includegraphics[width=0.65\linewidth]{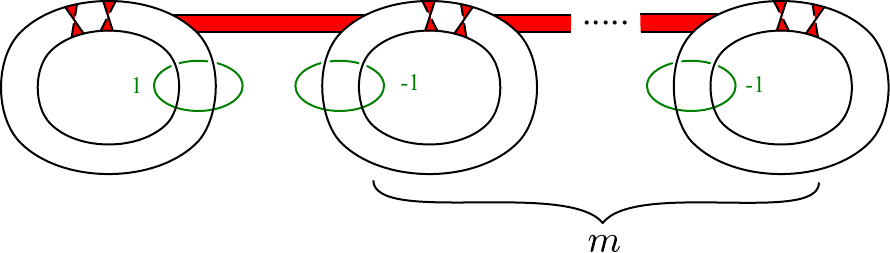}
    \caption{Banded unlink diagram of $\Sigma_2$}
    \label{fig:Sigma_d}
\end{figure}

\begin{lem}\label{lem:coloring_of_cntsum_with_cpxcurve}
	Let $S$ be a surface knot in $S^4$, and let $K$ be a finite kei. Then
	$$
	\Col_{K} (S\#\Sigma_d)=
	\begin{cases}
		\Col_{K}(S) & d\ \text{is even}\\ \\
		\{\text{constant coloring}\} & d\ \text{is odd}
	\end{cases}
	$$
\end{lem}

\begin{proof}
    Since taking connected sum of degree-$d$ algebraic curves and its mirroring does not change the relation that gives from the connected sum, $K(S\#\Sigma_d)$ and $K(S\#C_d)$ are isomorphic for every $d$. By Corollary~\ref{cor:K(S+C_d)}, when $d$ is even, $K(S\#\Sigma_d)$ is isomorphic to $K(S)$, so $\Col_{K}(S\#\Sigma_d)=\Col_{K}(S)$; when $d$ is odd, $K(S\#\Sigma_d)$ is isomorphic to $\langle x\rangle$, so $\Col_{K}(S\#\Sigma_d)$ is the constant coloring.
\end{proof}

\begin{lem}\label{lem:bridge_number}
	Let $S$ be a surface knot in $S^4$ satisfying $b(S)= 3\log_{\#K}(\#\Col_K(S))-\chi(S)$. Then $b(S\#\Sigma_d)=b(S)+(m+1)(d-1)(d-2)$ holds when $d$ is even.
\end{lem}

\begin{proof}
	The bridge number of a degree-$d$ algebraic curve is $(d-1)(d-2)+1$
	\cite{lambert2022bridge}.
	Taking connected sum of $S$ with $\Sigma_d$, we obtain
	\begin{equation}\label{eq:bridge_upper_bound}
		b(S\#\Sigma_d)\leq b(S)+(m+1)(d-1)(d-2).
	\end{equation}
	On the other hand, since
	$$b(S)= 3\log_{\#K}(\#\Col_K(S))-\chi(S)$$
	and $d$ is even, Lemma~\ref{lem:coloring_of_cntsum_with_cpxcurve} implies
    \begin{equation}\label{eq:bridge_lower_bound}
        \begin{aligned}
        	b(S\#\Sigma_d)
        	&\geq 3\log_{\#K}(\#\Col_K(S\#\Sigma_d))-\chi(S)+(m+1)(d-1)(d-2) \\
        	&= b(S)+(m+1)(d-1)(d-2).
        \end{aligned}
    \end{equation}
	By combining \eqref{eq:bridge_upper_bound} and
	\eqref{eq:bridge_lower_bound}, we conclude that
	$$b(S\#\Sigma_d)=b(S)+(m+1)(d-1)(d-2).$$
\end{proof}

For a classical knot $k$ in $S^3$, let $\tau^{n}(k)$ denote the $n$-twist spin of $k$,
introduced by Zeeman~\cite{zeeman1965twisting}.
It is obtained by removing a small open ball from $S^3$ so that $k$ becomes a properly embedded arc, spinning this arc around the boundary $S^2$ in $S^4$, and inserting $n$ full twists along the meridional direction during the spinning process. The resulting surface $\tau^{n}(k)$ is a smoothly embedded $2$-sphere in $S^4$.
We next consider the connected sums of the twist-spun knots in $S^4$ and $\Sigma_d$.

In \cite{sato2022bridge}, the bridge numbers of twist-spun knots were computed:

\begin{thm}[\cite{sato2022bridge}]\label{thm:bg_num_tsp}
    Let $k$ be a classical knot in $S^3$, and let $t_k$ be the tangle diagram obtained by removing an arc from $k$.
    Let $P$ denote the real projective plane $\mathbb{RP}^2$, and let $T$ denote the torus in $S^4$.
    Suppose there exists a finite kei $K$ with $b(k)=\log_{\#K}(\# \Col_K(t_k))$, then for any $m$ in $ \Z$, the following equations hold:
	\begin{align*}
		&b(\tau^{2m}(k))=3b(k)-2,\\
		&b(\tau^{2m}(k)\#P)=3b(k)-1,\\
		&b(\tau^{2m}(k)\#T)=3b(k).
	\end{align*}
\end{thm} 

Applying Lemma~\ref{lem:bridge_number}, we have the next corollary: 
\begin{cor}\label{cor:bridge_num_tau2mdk}
    Let $\tau^{2m}_d(k)$ denote the connected sum $\tau^{2m}(k)\#\Sigma_d$. Then we have 
    \begin{align*}
        &b(\tau^{2n}_d(k))=3b(k)-2+(m+1)(d-1)(d-2)\\
        &b(\tau^{2n}_d(k)\#P)=3b(k)-1+(m+1)(d-1)(d-2)\\
        &b(\tau^{2n}_d(k)\#T)=3b(k)+(m+1)(d-1)(d-2).
    \end{align*}
\end{cor}

\begin{lem}\label{lem:coloring_of_t_p,q}
        Let $q,q'$ be two integers such that $q'<q$. Then we have the following inequality:
        $$\#\Col_{R_q}(\#^{l}t_{2,q'})< q^{l+1}.$$
    \end{lem}

    \begin{proof}
    A direct calculation on the standard diagram of the $(2,q')$-torus knot shows that
    $$\#\Col_{R_q}(t_{2,q'})=q\,\gcd(q,q')<q^2.$$
    By using the formula  of the connected sum 
    \begin{equation}\label{eq:formula _connected_sum}
        \#\Col_{R_q}(k_1\#k_2)=\frac{\#\Col_{R_q}(k_1)\,\#\Col_{R_q}(k_2)}{q},
    \end{equation}
    where $k_1$ and $k_2$ are classical knots in $S^3$, we obtain
    $$\#\Col_{R_q}(\#^{l}t_{2,q'})
    =\frac{\bigl(\#\Col_{R_q}(t_{2,q'})\bigr)^l}{q^{l-1}}
    < q^{l+1}.$$
    \end{proof}

\begin{proof}[Proof of Theorem~\ref{thm:infinitely_many_bridge}]
	We take $d=2$, $m=0$. Let the classical knot be $\#^{l}t_{2,q}$, where $t_{2,q}$ denotes the $(2,q)$-torus knot and $q$ is an integer. 
    According to \cite{asami2009colorings}, $\#\Col_{R_q}(t_{2,q})= q^2$
    Using formula~\eqref{eq:formula _connected_sum}, we obtain that $\#\Col_{R_q}(\#^lt_{2,q})=q^{l+1}$. It follows that 
    $$b(\#^lt_{2,q})=\log_{\#{R_q}}\bigl(\#\Col_{R_q}(\#^lt_{2,q})\bigr)=l+1.$$
	By Corollary~\ref{cor:bridge_num_tau2mdk}, we have 
    \begin{align*}
        &b(\tau^0_2(\#^{l}t_{2,q}))=3l+1,\\
        &b(\tau^0_2(\#^{l}t_{2,q})\#P)=3l+2,\\
        &b(\tau^0_2(\#^{l}t_{2,q})\#T)=3l+3.
    \end{align*}
	For each $l$, we have 
    \begin{equation}\label{eq:col_twist=origin}
        \#\Col_{R_{q}}(\tau^{0}_{2}(\#^lt_{2,q}))=\#\Col_{R_{q}}(\#^lt_{2,q})=q^{l+1}.
    \end{equation}
    Take integer $q'$ with $q'<q$. According to Lemma~\ref{lem:coloring_of_t_p,q}, 
    \[
    \#\Col_{R_q}(\tau^0_d(\#^{l}t_{2,q'}))=\#\Col_{R_q}(\#^{l}t_{2,q'})< q^{l+1},
    \]
    so $\tau^0_d(\#^{l}t_{2,q})$ and $\tau^0_d(\#^{l}t_{2,q'})$ are not isotopic.
    Then we consider the cases of the connected sums with the projective plane $P$ and the torus $T$. 
    Consider the banded unlink diagrams of $P$ and $T$ shown in Figure~\ref{fig:P_and_T}. 
    
    The standard unknotted projective plane $P$ corresponds to a diagram with one generator $x$ and a twisted band relation which implies $x^2\equiv1$ (or equivalent involutory relation). This relation is automatically satisfied by the operation of kei. Thus, the fundamental kei of the connected sum $K(\tau^0_2(\#^{l}t_{2,q}) \# P)$ is isomorphic to $K(\tau^0_2(\#^{l}t_{2,q}))$.
    Similarly, the standard torus $T$ does not induce any non-trivial relations in the kei structure, and $K(\tau^0_2(\#^{l}t_{2,q}) \# T)$ is isomorphic to $K(\tau^0_2(\#^{l}t_{2,q}))$.
    
    It follows that the number of colorings remains invariant:
    $$ \#\Col_{R_q}(\tau^0_2(\#^{l}t_{2,q}) \# P) = \#\Col_{R_q}(\tau^0_2(\#^{l}t_{2,q})) = q^{l+1}, $$
    and similarly for $\# T$:
    $$ \#\Col_{R_q}(\tau^0_2(\#^{l}t_{2,q}) \# T) = \#\Col_{R_q}(\tau^0_2(\#^{l}t_{2,q})) = q^{l+1}. $$

    Now we apply the same argument as in the sphere case. Since we have established that $\#\Col_{R_q}(\tau^0_2(\#^{l}t_{2,q'})) < q^{l+1}$ (where $q' < q$), the same inequality holds for the stabilized surfaces:
    $$ \#\Col_{R_q}(\tau^0_2(\#^{l}t_{2,q'}) \# P) = \#\Col_{R_q}(\tau^0_2(\#^{l}t_{2,q'}) \# T) = \#\Col_{R_q}(\tau^0_2(\#^{l}t_{2,q'})) < q^{l+1}. $$
    
    Comparing the two results:
    \begin{align*}
        \#\Col_{R_q}(\tau^0_2(\#^{l}t_{2,q}) \# P) > \#\Col_{R_q}(\tau^0_2(\#^{l}t_{2,q'}) \# P). \\
        \#\Col_{R_q}(\tau^0_2(\#^{l}t_{2,q}) \# T) > \#\Col_{R_q}(\tau^0_2(\#^{l}t_{2,q'}) \# T). 
    \end{align*}
   
    Thus, we conclude:
    $$
    \tau^0_2(\#^{l}t_{2,q})\# P \ncong \tau^0_2(\#^{l}t_{2,q'})\# P,
    \qquad
    \tau^0_2(\#^{l}t_{2,q})\# T \ncong \tau^0_2(\#^{l}t_{2,q'})\# T.
    $$
    Therefore, by varying $q$, we obtain infinitely many pairwise non-isotopic surfaces with the same bridge number.
\end{proof}

\begin{figure}
    \centering
    \includegraphics[width=0.5\linewidth]{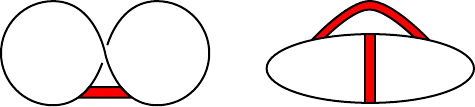}
    \caption{Banded unlink diagrams of $P$ (left) and $T$ (right)}
    \label{fig:P_and_T}
\end{figure}


\subsection{Surface knots in $4$-manifolds with the common knot groups}

In this section, we study surface knots with isomorphic knot groups but distinct fundamental quandles. In $S^4$, Tanaka and Taniguchi \cite{tanaka20252} showed that twist-spun knots can be distinguished by their quandle type even when their knot groups are isomorphic. As a direct consequence of Theorem~\ref{thm:qdle_presentation} and Corollary~\ref{cor:local_surface}, this result extends to local surface knots in arbitrary $4$-manifolds without 
$1$-handles:

 \begin{cor}\label{cor:quandle_fund_group}
    Let $X$ be a $4$-manifold without $1$-handle. There exist infinitely many triples of pairwise non-isotopic local surfaces $\{S_1,S_2,S_3\}$ embedded in $X$ such that 
	\begin{enumerate}[(1)]
		\item the knot groups $G(S_1), G(S_2), G(S_3)$ are isomorphic to each other, 
		\item the knot quandles $Q(S_1),Q(S_2),Q(S_3)$ are pairwise non-isomorphic. 
	\end{enumerate}
\end{cor}

\begin{proof}
    Let $p, q, r$ be pairwise distinct integers greater than $1$, and let
    \[S_1 = \tau^p(t_{q,r}), \quad S_2 = \tau^q(t_{r,p}), \quad S_3 = \tau^r(t_{p,q}).\]
    Take a $4$-ball $B^4$ which contains the surface links $S_1, S_2, S_3$ in $X$. Embed this $B^4$ in $S^4$, and denote the corresponding surfaces in $S^4$ by $S'_1, S'_2, S'_3$. By Corollary~\ref{cor:local_surface}, $Q(S_i)$ and $Q(S'_i)$ are isomorphic for each $i$. By the same discussion as \cite{tanaka20252}, we can prove that $G(S_1), G(S_2), G(S_3)$ are isomorphic, and $Q(S_1), Q(S_2), Q(S_3)$ are pairwise non-isomorphic. By changing the choice of $p, q, r$, we obtain infinitely many triples of such $(S_1, S_2, S_3)$.
\end{proof}

We further construct explicit examples of surfaces representing non-trivial homology classes in $\CP^2\#m\overline{\CP}^2$, which are in particular not local, and show that the same phenomenon occurs in this setting:

\begin{thm}\label{thm:infinite_qdl_CP2}
	For every integer $m\geq0$, there exist infinitely many triples of pairwise homologous surfaces $\{S_1,S_2,S_3\}$ embedded in $\CP^2\#m\overline{\CP}^2$ such that 
	\begin{enumerate}[(1)]
		\item the knot groups $G(S_1), G(S_2), G(S_3)$ are isomorphic to each other, 
		\item the knot quandles $Q(S_1),Q(S_2),Q(S_3)$ are pairwise non-isomorphic. 
	\end{enumerate}
\end{thm}

For any $x, y$ in $Q$ and $n$ in $\mathbb{Z}$, let $x^{y^n}$ denote the element obtained by applying the right operation by $y$ on $x$, $n$ times. Define the \textit{type} of $Q$ by 
$$\Type(Q) := \min \{ n \in \mathbb{N} \mid y^n\equiv1 \ \text{for any} \ y \in Q \},$$
with the convention that $\min \emptyset = \infty$.

In \cite{zeeman1965twisting}, Zeeman showed that the $n$-twist spun knot $\tau^n(k)$ is a fibered $2$-knot for every $n>0$, whose fiber is a once-punctured copy of $M^n(k)$, the $n$-fold cyclic branched cover of $S^3$ along $k$. 
In $S^4$, the following theorem of Tanaka and Taniguchi shows that twist–spun knots can be distinguished by their quandle type.

\begin{thm}[\cite{tanaka20252}]\label{thm:twist_spun_type}
	For $i\in \{1,2\}$, let $k_i$ be a non-trivial classical knot in $S^3$, and $n_i\in \N$ is greater than $1$. Then we have the followings: 
	\begin{enumerate}[(1)]
		\item $\Type(Q(\tau^{n_i}(k_i))=n_i$,
		\item If $Q(\tau^{n_1}(k_1))\cong Q(\tau^{n_2}(k_2))$, then we have $n_1=n_2$. 
	\end{enumerate}
\end{thm}

\begin{lem}\label{lem:twist_spun+Sigmad}
    Let $\tau_{mn}^n(k)=\tau^n(k)\# \Sigma_{mn}$, where  $\Sigma_{mn}$ denotes the surface defined previously, and let $\Sigma_d = C_d \#^{m} \overline{C}_d$ be the $m$-fold connected sum of the degree $d$ algebraic curve with its complex conjugate. Then for any $m\in\mathbb N$,
    \[
    Q(\tau_{mn}^n(k)) \cong Q(\tau^n(k)).
    \]
\end{lem}

\begin{proof}
	Let $D$ be a banded unlink diagram of $\tau^n(k)$. Then the quandle $Q(\tau^n(k))$ admits a presentation $\langle x_1,\dots,x_s \mid r_1,\dots ,r_t\rangle$, where each $x_i$ is the primary generator associated with an arc of $D$, and each $R_i$ is the primary relator determined by a crossing. Since $type(Q(\tau^{n}(k))=n$, we may impose additional operator relators and rewrite the presentation in the form $[x_{1},\dots,x_{s}\ |\ r_{1},\dots,r_{t},x_{i}^n\equiv1 \ \forall i\in \{ 1,\dots,s \}]$. 
	
	Now consider a banded unlink diagram $D'$ obtained by taking the connected sum of the diagrams of $\tau^n_{d}(k)$ and $\Sigma_d$. Take $d=mn\ (m\in \N)$, then $Q(\tau^n_{d}(k))$ has a presentation 
    \[
    [x_{1},\dots,x_{s}\ |\ r_{1},\dots,r_{t},x_{i}^n\equiv 1 \ \forall i\in \{ 1,\dots,s \}, x_{k}^d\equiv1],\]
    where $x_k$ is the arc that we applied the connected sum. By Lemma \ref{lem:=1}, $Q(\tau^n_{d}(k))$ can be further rewritten in 
    \[
     [x_{1},\dots,x_{s}\ |\ r_{1},\dots,r_{t},x_{i}^n\equiv1 ,x_{i}^d\equiv1\ \forall i\in \{ 1,\dots,s \}].\]
    Since $d=mn$, the relation $x_i^{mn}\equiv 1$ follows from $x_{i}^n \equiv 1$, and it may be removed from the presentation by applying Tietze move. Then $Q(\tau_{mn}^n(k))\cong Q(\tau^n(k))$. 
\end{proof}

\begin{lem}\label{lem:CP2_twist_spun_type}
For $i \in \{1,2\}$, let $k_i$ be a non-trivial classical knot in $S^3$, and let $n_i$ be integer greater than $1$. 
Suppose that both $n_1$ and $n_2$ divide $d$. 
If $Q(\tau_d^{n_1}(k_1))$ and $Q(\tau_d^{n_2}(k_2))$ are isomorphic, then $n_1 = n_2$.
\end{lem}

\begin{proof}
	By Lemma \ref{lem:twist_spun+Sigmad}, $Q(\tau_d^{n_1}(k_1))\cong Q(\tau^{n_1}(k_1))\cong Q(\tau^{n_2}(k_2))\cong Q(\tau_d^{n_2}(k_2))$. Then $n_1=n_2$ holds by Theorem \ref{thm:twist_spun_type}. 
\end{proof}

Finally let us prove the main theorem in this section.

\begin{proof}[Proof of Theorem \ref{thm:infinite_qdl_CP2}] 
	Let $p,q,r$ be pairwise distinct integers greater than $1$. Take $d=pqr$. Let
    \begin{align*}
    S_1 =\tau^p_d(t_{q,r})\quad S_2=\tau^q_d(t_{r,p})\quad S_3 =\tau^r_d(t_{p,q}),
    \end{align*}
    where $t_{q,r},t_{r,p},t_{p,q}$ are $(q,r),(r,q),(p,q)$-torus knot, respectively. By using Seifert--van Kampen's theorem, we have
    \begin{align*}
    G(S_1)& =G(\tau^p(t_{q,r}))*G(\Sigma_d), \\
    G(S_2)&=G(\tau^q(t_{r,p}))*G(\Sigma_d), \\
    G(S_3)&=G(\tau^r(t_{p,q}))*G(\Sigma_d).
    \end{align*}
    Gordon \cite{gordon1972twist} proved that the knot group of the twist–spun torus knot $\tau^{r}(t_{p,q})$ is isomorphic to $\pi_{1}(M^{r}(t_{p,q})) \times \mathbb{Z}$, where $M^{r}(t_{p,q})$ is the $r$-fold cyclic branched cover of $S^3$ along $t_{p,q}$. 
    
    Furthermore, the $3$-manifolds $M^{p}(t_{q,r})$, $M^{q}(t_{r,p})$, and $M^{r}(t_{p,q})$ are known to be mutually homeomorphic \cite{milnor19753}. As a consequence, the knot groups associated to the twist–spun knots $G(\tau^{p}(t_{q,r}))$, $G(\tau^{q}(t_{r,p}))$, and $G(\tau^{r}(t_{p,q}))$ are mutually isomorphic, so $G(S_1)\cong G(S_2)\cong G(S_3)$. 
	Since $p, q, r$ are pairwise distinct, by Lemma \ref{lem:CP2_twist_spun_type}, the knot quandles $Q(S_1), Q(S_2), Q(S_3)$ are pairwise non-isomorphic. By changing the choice of $p,q,r$, we can give infinitely many triples of such $(S_1,S_2,S_3)$.
\end{proof}

\section*{Acknowledgement}
The author would like to express profound gratitude to Professor Hiraku Nozawa for his invaluable guidance, insightful comments, and continuous support throughout the development of this paper. The author is deeply grateful to Delphine Moussard, Sylvain Courte, and Qiuyu Ren for fruitful discussions on related topics; their generous sharing of ideas has greatly broadened the author's mathematical perspective and provided valuable inspiration. Special thanks go to Junpei Yasuda for his patience in answering numerous questions and offering helpful insights.

\bibliographystyle{alpha}
\bibliography{ref.bib}

\end{document}